\documentclass[12pt,a4paper]{article}
\usepackage{latexsym,amssymb,amsmath,a4wide,theorem}
\usepackage[a4paper, total={6in, 10in}]{geometry}
\usepackage[isolatin]{inputenc}
\usepackage[T1]{fontenc}
\usepackage{times}
\usepackage{cite}
\usepackage{amsfonts}
\usepackage{graphicx}
\usepackage[all]{xy}
\usepackage{esint}
\usepackage[titletoc,title]{appendix}
\usepackage{indentfirst}
\usepackage{float}
\usepackage{cancel}
\usepackage{booktabs}
\usepackage{makecell}
\usepackage{color}

\def\sqr#1#2{{\vbox{\hrule height.#2pt
     \hbox{\vrule width.#2pt height#1pt \kern#1pt
           \vrule width.#2pt}
     \hrule height.#2pt}}}

\usepackage{mathrsfs}

\newtheorem{Thm}{Theorem}[section]

\newtheorem{Lem}{Lemma}[section]

\newtheorem{Rek}{Remark}[section]
\newtheorem{Def}{Definition}[section]

\numberwithin{equation}{section}

\title{{Existence and Multiplicity of Normalized Solutions for Dirac Equations with non-autonomous  nonlinearities}}

\author{
}

\date{}

\begin{document}
\author{ {~~Anouar Bahrouni$^{a}$, Qi Guo$^{b}$, Hichem Hajaiej$^{c}$ and Yuanyang Yu$^{d}$\footnote{corresponding author: yyysx43@163.com}}\\
\small Mathematics Department, University of Monastir, Faculty of Sciences, 5019 Monastir, Tunisia$^{a}$\\
\small School of Mathematics, Renmin University of China, Beijing, 100872, P.R. China$^{b}$\\
\small Department of Mathematics, California State University, Los Angeles,\\
\small 5151 State University Drive,
Los Angeles, CA 90032, USA$^{c}$\\
\small Department of Mathematical Sciences, Tsinghua University,
Beijing, 100084,  P.R. China$^{d}$ } \maketitle

\noindent{\bf Abstract}\quad In this paper, we study the following
nonlinear Dirac equations
\begin{align*}
\begin{cases}
-i\sum\limits_{k=1}^3\alpha_k\partial_k u+m\beta u=f(x,|u|)u+\omega u,\\
   \displaystyle \int_{\mathbb{R}^3} |u|^2dx=a^2,
    \end{cases}
\end{align*}
where $u: \mathbb{R}^{3}\rightarrow \mathbb{C}^{4}$, $m>0$ is the
mass of the Dirac particle, $\omega\in \mathbb{R}$ arises as a
Lagrange multiplier, $\partial_k=\frac{\partial}{\partial x_k}$,
$\alpha_1,\alpha_2,\alpha_3$ are $4\times 4$ Pauli-Dirac matrices,
$a>0$ is a prescribed constant, and $f(x,\cdot)$ has several
physical interpretations that will be discussed in the Introduction.
Under general assumptions on the nonlinearity $f$, we prove the
existence  of $L^2$-normalized solutions for the above nonlinear
Dirac equations by using perturbation methods in combination with
Lyapunov-Schmidt reduction. We also show  the multiplicity of these
normalized solutions thanks to the multiplicity theorem of Ljusternik-Schnirelmann. Moreover, we obtain bifurcation results of this problem. \\

\numberwithin{equation}{section}

\tableofcontents

\section{Introduction}
This section is divided into two subsections. The first part is
dedicated to the physical motivation of the problem, while the
second part discusses the main results of this paper.
\subsection{ Physical motivation of normalized solutions for Dirac equations}
In this subsection, we focus our attention on the physical
motivation of the existence of normalized solutions for Dirac
equations. We will then point out the complexity and the challenging
aspects of this problem.

\medskip
Dirac derived the now called Dirac operator starting from the usual
classical expression of the energy of a free relativistic particle
of momentum $p\in \mathbb{R}^3$, and mass $m$,
\[E^2=c^2|p|^2+m^2c^4,\]
where $c$ is the speed of light. By means of the usual
identification
\[p\leftrightarrow -i\hbar \nabla,\]
where $\hbar$ is Planck's constant. The Dirac operator is of  first
order and has the following form:
\begin{equation}\label{dir}
D_c=-ic\hbar \alpha\cdot \nabla +mc^2 \beta ,
\end{equation}
  where
$\alpha\cdot\nabla=\sum\limits_{k=1}^3 \alpha_k\partial_k$, and
$\partial_k=\frac{\partial}{\partial x_k}$, $\alpha_1$, $\alpha_2$,
$\alpha_3$ and $\beta$ are $4\times 4$ Pauli-Dirac matrices (in
$2\times 2$ blocks):
 \begin{align*}
\alpha_k=\begin{pmatrix}
0 &\sigma_k\\
\sigma_k &0
\end{pmatrix},\beta=\begin{pmatrix}
I &0\\
0 &-I
\end{pmatrix},
\end{align*}
with
\begin{align*}
\sigma_1=\begin{pmatrix}
0 &1\\
1&0
\end{pmatrix}, \sigma_2=\begin{pmatrix}
0 &-i\\
i&0
\end{pmatrix}, \sigma_3=\begin{pmatrix}
1&0\\
0&-1
\end{pmatrix}.
\end{align*}
The time-dependent Dirac equation describing the evolution of a free
particle is given by: \begin{equation}\label{time} i\hbar
\frac{\partial}{\partial t} \Psi(t,x)=D_c\Psi. \end{equation} This
equation has been successfully used in physics to describe the
relativistic fermions, see \cite{Thaller1992book}.
\medskip

By a direct computation, we can prove that $D_1^{2}=-\Delta+m^{2}$,
and therefore equation \eqref{time} can be regarded as a special
case of infinite dimensional Hamiltonian systems, which are called
quantum mechanical systems, see \cite{Dong-Ding-Guo}.
Hence, all the basic theorems about Hamiltonian systems do apply.
If $c=\hbar=1$, the Lagrangian density that describes the local
interaction of the nonlinear Dirac field is given by
\begin{equation}\label{F}
\mathcal{L}_{ND}=\bar{\psi}\left(i \gamma^0\cancel{D}-m \right)\psi
+ F(\bar{\psi},\psi),
\end{equation}
where $\cancel{D}= \partial_t+\displaystyle \sum_{k=1}^{3}\alpha_k
\partial_k$, $\gamma^0=\beta$ and $\bar{\psi}$ represents the Dirac adjoint
of $\psi$.
The external fields  $F(x,\cdot)$ in \eqref{F} arise in mathematical
models of particle physics, especially in nonlinear topics. It is
inspired by approximate descriptions of the external forces that
only involve functions of fields. The nonlinear self-coupling $F(\cdot ,\cdot)$, that describes a self-interaction in quantum electrodynamics,
gives a closer description of many particles found in the real
world, see \cite{ranada}. M. Soler \cite{Soler}, who was the first to investigate the
stationary states of the nonlinear Dirac field with the scalar
fourth order self coupling, proposed \eqref{F} as a model of
elementary extended fermions.
The corresponding Dirac equation was a key development in harmonizing
the Schr\"odinger equation with special relativity, see \cite{ranada}. It is
interesting to note that in the realm of Schr\"odinger equations,
nonlinearities are often observed as approximations in optics and
condensed matter fields. At the quantum mechanics frontier, it is
expected that quantum nonlinearities can be detected at high energy
levels or extraordinarily short distances, with a potential link to
phenomena such as neutrino oscillations, see \cite{ng}. The nonlinear Dirac
equations play a pivotal role in several scientific disciplines,
including atomic, nuclear, and gravitational physics, see
\cite{lazarids,Bjorken,ng}. These equations offer a comprehensive
theoretical framework that portrays the behavior of fermions under
intense gravitational forces, thereby providing deep insights into
the properties of matter in extreme conditions. Furthermore, they
are also employed in condensed matter physics, where they
effectively illustrate the behavior of electrons in specific
materials.
\medskip

Denoting $\mathcal{S}=\displaystyle\int\mathcal{L}_{ND}dxdt$ the total action,
and making the variation of $\mathcal{S}$ with respect to $\delta
\psi$ or $\delta \bar{\psi}$, we get the following nonlinear Dirac
equation
\begin{equation}\label{S}
i\cancel{D}\psi-m \psi+F_{\bar{\psi}}(\bar{\psi},\psi)=0.
\end{equation}
Several novel nonlinear photonic systems currently explored are
modeled by Dirac equations \eqref{S}. Examples include fiber Bragg
gratings \cite{eggleton}, dual-core photonic crystal fibers
\cite{betlej}, and discrete binary arrays, which refer to systems
built as arrays coupled of elements of two types. Earlier
experimental work on binary arrays has already shown the formation
of discrete gap solitons \cite{morandotti}. Three of the many novel
examples that have been recently considered are: a dielectric
metallic waveguide array \cite{aceves}; an array of vertically
displaced binary waveguide arrays with longitudinally modulated
effective refractive index \cite{marini}, and arrays of coupled
parity-time  nanoresonators \cite{lazarids}. Moreover, the nonlinear
Dirac equation \eqref{S} with power nonlinearity has some
applications in spin geometry, such as the prescribed mean curvature
problem,
  see \cite{ammann1}.
  \medskip

The linear Dirac equation \eqref{time} is invariant under wave
function rescaling, a feature in quantum mechanics that allows the
normalization of solutions.
 This normalization is not simply beneficial, but sometimes required for the interpretation of experimental observations. However, nonlinear terms generally do not
  exhibit this scaling invariance. As such, we need to introduce an additional constraint, namely, $$\displaystyle \int_{\mathbb{R}^3} |\psi(t,x)|^2 dx=a^2,$$ to find normalized
  solutions for the nonlinear Dirac equation, where $a>0$ is a prescribed constant. This constraint ensures the validity and reliability of our solutions, thereby advancing our
  understanding of nonlinear Dirac equations.\medskip

   We consider the standard
solitary wave solutions (stationary states) of \eqref{S}, which have
the form
\begin{equation}\label{so}
\psi(t,x)=e^{-i\omega t}u(x).
 \end{equation} Solitary wave solutions
of nonlinear Dirac problem \eqref{S} are considered as particle-like
solutions. These solutions are in some
sense solitons that propagate without changing shape. To simplify the nonlinear model, we assume that the nonlinear term is gauge invariant and set $F:\mathbb{R}^3\times \mathbb{C}^4\rightarrow \mathbb{R}$ satisfying
$F_u(x,u)=\gamma^0
F_{\bar{\psi}}(\bar{u},u)$, then the nonlinear Dirac equation \eqref{S} reads as follows:
\begin{equation}\label{stat}
-i\alpha\cdot\nabla u+m\beta u-\omega u=F_u(x,u).
 \end{equation}

In the past two decades, there has been an increasing interest in
the study of existence of solitary wave solutions for nonlinear
Dirac equations \eqref{stat}.  In \cite{MR2434346,MR847126,
MR968485}, the authors studied the problem \eqref{stat} with  a
nonlinearity of Soler-type and $\omega\in (0,m)$.
 In \cite{MR1071235}, Balabane, Cazenave and V\'{a}zquez showed the existence of stationary
  states with compact support for problem \eqref{stat} with singular nonlinearities.  As a continuation of the last work, Treust  in \cite{MR3070757} studied the existence of nodal
  solutions.
 Then variational methods were applied to search for solutions by Esteban and S\'er\'e for the case  $\omega\in (0,m)$,
 see \cite{Esteban-Sere1995CMP, MR1897689,MR2434346}. 
   Other relevant results can be found in
 \cite{Bartsch-Ding2006JDE,Ding-Ruf2008ARMA,MR2450893}. Another way to find the solutions of nonlinear Dirac equations
 is the use of perturbation theory, see  \cite{MR1750047}.
 \medskip

A particular class of solitary waves is when $u$  given by \eqref{so}
is a normalized solution, meaning that it solves a constrained
minimization problem associated with \eqref{stat}. These types of
solutions have attracted the attention of many research groups
during the last few years. For the classical PDEs studied so far,
the normalized solutions have the best profile as they very often
enjoy nice symmetry, stability, asymptotic, and decay
features. In our context, several photonic systems are modeled by
the Dirac equations under study. The corresponding normalized
solutions can be implemented using photonic states in optical fibers
without excessive quantum complexity. The
normalized solutions are also the best candidates to provide a
simple stable realization of higher-dimensional quantum entanglement
used in cryptography, which is fundamendal to obtain higher security
margin and an increased level of noise at a given level of security,
see \cite{cerf}. Additionally, the normalized solutions of the first
class of Dirac equations under study can be seen as the
representation of the number of particles in a Bose-Einstein
condensate, also known as the fifth matter, see
\cite{Thaller1992book}.

\medskip
Despite the great importance of normalized
solutions in many fields as one can easily tell from the long list
of valuable contributions addressing different aspects of the
normalized solutions (see
\cite{bahrouni,bartsch,bartsch2,Bartsh-Zou,hichem,hs1,hs2,jeanjean,Soave1}),
the literature is quite silent about these solutions for  Dirac
equations. More precisely so far, there are only three contributions
\cite{DingDirac, Nolasco,Zelati} dealing with very particular
nonlinearities.
The situation for Dirac equations is much more challenging than in
all the previous PDEs handled.
 The main difficulty comes from the fact that the corresponding energy functionals
are strongly indefinite since the spectrum of Dirac operators
consists of positive and negative continuous spectrum. Because of that all existing methods do not apply to this setting. Even the recent general approach developped by the third author and L. Song \cite{GHS} seems not to apply to the current setting.
\cite{Buffoni} was the first work devoted to the study of the
existence of normalized solutions for classical elliptic problems in
the complex setting of indefinite functionals  by using a
penalization method in the spirit of \cite{Esteban}. In
\cite{Nolasco}, Nolasco obtained the existence of $L^{2}-$normalized
solitary wave solution for the Maxwell-Dirac equations in
$(3+1)-$Minkowski space. The author used a variational
characterization for critical points of the energy functional,
inspired by the one used to characterize the first eigenvalue of the
Dirac operators with Coulomb-type potentials, see \cite{dolbeaut}.
 \cite{DingDirac} was the first and the only work devoted
to study the existence of normalized solutions for the Dirac problem
\eqref{S}. More precisely, the authors dealt with a particular
nonlinearity and assumed that $F_u(x,u)=K(x) |u|^{p-2}u$, where $p\in
(2,8/3)$ and $K$ is a weight function used to ensure compactness.
The authors presented a new approach that is based on perturbation
arguments. Lastly, let us mention a recent work about the
existence of normalized solution of the linear coupled
Klein-Gordon-Dirac system which describes the interaction
electron-photon analyizing the Euler-Lagrange equations for a system
consisting of a spinor field coupled with a massless scalar field,
see \cite{Zelati}. To the best of our knowledge, there are no
results studying the existence of multiple normalized solutions of
Dirac equations \eqref{stat} with non-autonomous nonlinearities.
This paper aims to fill this gap by giving a new approach and
technics.
 \subsection{Main results}
  In this paper, we focus our attention on the existence and multiplicity of the normalized solutions for the
following nonlinear Dirac equation
\begin{align}\label{eq1.1}
\begin{cases}
    -i\alpha\cdot\nabla u+m\beta u=f(x,|u|)u+\omega
    u\\
     \displaystyle \int_{\mathbb{R}^3} |u|^2dx=a^2,
    \end{cases}
\end{align}
where $a>0$ is a prescribed constant.\medskip

To state our main results, we need the following assumptions:
\begin{itemize}
        \item[$(f_1)$] $f(x,0)=0$ and $f(x,\cdot)\in \mathcal{C}^1(\mathbb{R}^+,\mathbb{R}^+)$ for almost all $x\in\mathbb{R}^3$;
    \item[$(f_2)$]  $f(x,t)>0$  for all $t\in (0,\infty)$ and almost all $x\in \mathbb{R}^3$;
    \item[$(f_3)$] There exist $p,q\in\mathbb{R}$ with
    $2< p \leq q<3$ such that for almost all $x\in\mathbb{R}^3$, the function $t\mapsto \displaystyle\frac{f(x,t)}{t^{q-2}}$(resp. $t\mapsto \displaystyle\frac{f(x,t)}{t^{p-2}}$) is nonincreasing(resp. nondecreasing) on $(0,\infty)$;
    \item[$(f_4)$] $r(x):=f(x,1)\in L^\infty(\mathbb{R}^3)$ and
    \begin{equation*}
\lim_{R\to \infty}\operatorname*{ess\,sup}\limits_{|x|\geq R}r(x)=0;
    \end{equation*}
\item[$(f_5)$]
There exist positive constants $L, t_0$, $\alpha\in \left(0,8/3\right)$
and $\tau\in \left(0,(8-3\alpha)/2\right)$ such that
\begin{equation*}
    F(x,t) \geq L|x|^{-\tau} t^\alpha~~\text{a.e. on}~~S~~\text{and}~~0\leq t\leq t_0,
\end{equation*}
where $F(x,t)=\displaystyle\int_0^tf(x,\tau)\tau d\tau$ and
$S=\left\{tx\mid t\geq 1~\text{and}~x\in B_d(x_0)\right\}$ for some
$x_0\in \mathbb{R}^3$ and $0<d<|x_0|$.
\end{itemize}
\begin{Rek}\label{Rek1.1}
$(i)$   Note that, the conditions $(f_1)$-$(f_4)$ imply that, for
almost all $x\in\mathbb{R}^3$ and $t> 0$
        \begin{equation}\label{eq1.2}
                        \begin{split}
&0<(p-2)f(x,t)\leq f^\prime(x,t)t\bigg(=\frac{\partial f(x,t)}{\partial t}t\bigg)\leq (q-2)f(x,t),\\
&\frac{1}{q}f(x,t)t^2\leq F(x,t)\leq \frac{1}{p}f(x,t)t^2,
            \end{split}
    \end{equation}
and
    \begin{equation}\label{eq1.3}
            \begin{split}
&s^pF(x,t)\leq F(x,st)\leq s^qF(x,t) ~~\text{if}~~s\geq 1,\\
&s^qF(x,t)\leq F(x,st)\leq s^pF(x,t)~~\text{if}~~0<s\leq 1.
        \end{split}
\end{equation}
$(ii)$ Consequently, from $(i)$, we can deduce that
\begin{equation}\label{eq1.4}
                \begin{split}
    &r(x)s^p\leq F(x,s)\leq r(x)s^q~~\text{if}~~s\geq 1,\\
    &r(x)s^q\leq F(x,s)\leq r(x)s^p~~\text{if}~~0<s\leq 1,
                \end{split}
\end{equation}
for almost $x\in \mathbb{R}^{3}$.\\ $(iii)$ It is not difficult to
check, in view of $(i)$ and $(ii)$, that $F(x,\cdot)$ is strictly
convex and
    \begin{equation}\label{eq1.5}
F(x,t)\leq C\left(|t|^p+|t|^q\right),~~\forall t\in \mathbb{R}, \
\mbox{and~for~a.e.} \ x\in \mathbb{R}^{3}.
    \end{equation}
\end{Rek}

Now, we recall the following important definition.
\begin{Def}
 The number
$\omega_0$ is said to be a bifurcation point on the left of the equation
\eqref{eq1.1} if there exists
$$
\left\{\left(\omega_n, u_n\right)\right\} \subset\left\{(\omega,
u) \in \mathbb{R} \times
H^{1/2}\left(\mathbb{R}^3,\mathbb{C}^4\right)\mid -i\alpha \cdot
\nabla u+m\beta u=f(x,|u|)u+\omega u\right\}
$$
such that $\omega_n<\omega_0$ for all $n\in \mathbb{N}$, $\omega_n \rightarrow
\omega_0$, $\lim\limits _{n \rightarrow
\infty}\left\|u_n\right\|_{H^{1/2}}=0$.
\end{Def}

Our first result is:
\begin{Thm}\label{Thm1.1}
    Suppose that $(f_1)$-$(f_5)$ hold. Then, for any fixed $a>0 $ small enough, there is a solution $\left(\omega_a, u_a\right)$ of equation \eqref{eq1.1} such that
    \begin{equation*}
        \int_{\mathbb{R}^3}|u_a|^2dx=a^2.
        \end{equation*}
Moreover,
    $\omega_a$ is a bifurcation point of equation \eqref{eq1.1}.
\end{Thm}

Note that, Theorem \ref{Thm1.1} is the first result that addresses the
existence of normalized solutions for the Dirac equation \eqref{eq1.1}
with more general non-compact nonlinearities. Theorem \ref{Thm1.1}
can be seen as a generalization of the result obtained in
\cite{DingDirac}. Let us point out that the approach developed in
 \cite{DingDirac} is not applicable to obtain multiplicity results of normalized solutions. For this reason, we propose a
novel approach consisting in considering a new suitable constraint
set and exploiting some idea coming from
\cite{Buffoni1993,jeanjean1992}. Note that this new approach can be
applied to many problems that study the existence of normalized
solutions for nonlinear equations with indefinite associated
functionals.

\begin{Thm}\label{Thm1.2}
Suppose that $(f_1)$-$(f_5)$ hold. Then, for any fixed  $a>0$ small
enough, there is a sequence $\{(\omega_a^i,u_a^i)\}_{i=1}^\infty$ of
distinct solutions of equation \eqref{eq1.1} such that
    \begin{equation*}
    \int_{\mathbb{R}^3}|u_a^i|^2dx=a^2.
\end{equation*}
Moreover,
    $\omega_a^1, \omega_a^2,\cdots $  are bifurcation points on the left of equation \eqref{eq1.1}.
\end{Thm}

The existence of multiple normalized solutions for nonlinear
Schr\"odinger equations has been extensively investigated. Bartsch
and  Valeriola in \cite{bartsh1} derived infinitely many radial
solutions for a Schr\"odinger equations from a fountain theorem type
arguments.
 Ikoma and Tanaka \cite{Ikoma2} provided a
multiplicity result for Schr\"odinger equations and systems by
exploiting an idea related to symmetric mountain pass theorems, see
also \cite{jeanjeanlu,alveszamp}. 
Let us point out that
the critical point theorems used to prove the existence of multiple
normalized solutions in the above references are proposed for
positive defined functionals. Therefore the methods seem difficult to be
applied to Dirac equations since the spectrum of Dirac operators
consists of positive and negative continuous spectrum. To establish
Theorem \ref{Thm1.2}, we are going to  use the critical point
theorem \cite[Theorem 6.2]{jeanjean1992} in combination with the new
techniques developed in the proof of Theorem \ref{Thm1.1}. We would like to mention that, the bifurcation result obtained in Theorems \ref{Thm1.1} and \ref{Thm1.2} are new in the
topic of Dirac problems.\medskip

This paper is organized as follows. In Section $2$, we recall and
present some preliminary notions on the Dirac operator and give some
basic results that will be used to prove our main results. Sections
$3$ and $4$ are dedicated to the proof of Theorems \ref{Thm1.1} and
\ref{Thm1.2}.
\par
\vspace{3mm} {\bf Notation.~}Throughout this paper, we make use of
the following notations.
\begin{itemize}
    \item[$\bullet$]  $\mathbb{R}^+$ denotes the interval $[0,\infty)$.
     \item[$\bullet$] $\|\cdot\|_{L^q}$ denotes the usual norm of the space $L^q(\mathbb{R}^3,\mathbb{C}^4),1\leq q\leq\infty$;
    \item[$\bullet$]  $o_n(1)$ denotes $o_n(1)\rightarrow 0$ as $n\rightarrow\infty$;
    \item[$\bullet$]   $u\cdot v$ denotes the scalar product in $\mathbb{C}^4$ of $u$ and $v$, i.e., $u\cdot v=\sum\limits_{i=1}^4u_i\bar{v}_i$;
    \item[$\bullet$]  $C$ or $C_i(i=1,2,\cdots)$ are some positive constants may change from line to line.
\end{itemize}

\section{The functional-analytic setting}
In this section, we present the functional setting of problem
\eqref{eq1.1}.
Firstly, we give the spectral decomposition of the following Dirac operator
\begin{equation*}
    H_0:=-i \alpha\cdot \nabla+m\beta.
\end{equation*}
It is well-known that $H_0$ is a first
order, self-adjoint operator on $L^2(\mathbb{R}^3,\mathbb{C}^4)$ with
domain $\mathscr{D}(H_0) = H^1(\mathbb{R}^3,\mathbb{C}^4)$, then by
\cite[Lemma 3.3(b)]{Bartsch-Ding2006JDE}, we know that
\begin{equation*}
    \sigma(H_0)=\sigma_c(H_0)=\mathbb{R}\setminus (-m,m),
\end{equation*}
where $\sigma (H_0)$ and $\sigma_c(H_0)$ denote the spectrum and
continuous spectrum of operator $H_0$. Thus the space
$L^2(\mathbb{R}^3,\mathbb{C}^4)$ possesses the orthogonal
decomposition:
\begin{equation*}
    L^2=L^-\oplus L^+,~~~u=u^-+u^+,
\end{equation*}
so that $H_0$ is negative definite on $L^-$ and positive definite on
$L^+$. Let $E:= \mathscr{D}(|H_0|^{1/2})$ be the Hilbert
space with the inner product
\begin{equation*}
    (u,v)=\Re(|H_0|^{1/2}u,|H_0|^{1/2}v)_{L^2}
\end{equation*}
and the induced norm $\|u\|=(u,u)^{1/2}$, where
$|H_0|^{1/2}$ denote the square root  of the absolute value
of $H_0$ and $\Re$ stands for the real part of a complex number.
Since $\sigma(H_0)=\mathbb{R}\setminus(-m,m)$, one has
\begin{equation}\label{eq2.1}
m\|u\|_{L^2}^2\leq \|u\|^2,~\text{for~all}~u\in E.
\end{equation}
Note that the linear space $E$ corresponds to the fractional
Sobolev space $H^{1/2}(\mathbb{R}^3,\mathbb{C}^4)$ and the
norm of $E$ is equivalent to the usual norm of
$H^{1/2}(\mathbb{R}^3,\mathbb{R}^4)$ defined as follows:
\begin{equation*}
    \|u\|_{H^{1/2}}=\left(\int_{\mathbb{R}^3}\sqrt{1+|\xi|^2}|\hat{u}(\xi)|^2d\xi\right)^{1/2},
\end{equation*}
where $\hat{u}$ denotes the Fourier transform of $u$, which is
defined by
\begin{equation*}
    \hat{u}(\xi)=\frac{1}{(2\pi)^{3/2}}\int_{\mathbb{R}^3}e^{-ix\xi}u(x)dx.
\end{equation*}
 Hence, $E$ embeds continuously
into $L^q(\mathbb{R}^3,\mathbb{C}^4)$ for all $q\in [2,3]$, and
compactly into $L_{loc}^q(\mathbb{R}^3,\mathbb{C}^4)$ for all $q\in
[1,3)$. Moreover, $E$ can be decomposed as follows:
\begin{equation*}
    E=E^-\oplus E^+,~~~\text{where}~E^\pm=E\cap L^\pm
\end{equation*}
which is an orthogonal decomposition with respect to inner products
$(\cdot,\cdot)_{L^2}$ and $(\cdot,\cdot)$.\\
 Consider the following functional
 \begin{equation}\label{I}
 I(u)=\frac{1}{2}\|u^+\|^{2}-\frac{1}{2}\|u^-\|^{2}-\Psi(u),
\end{equation}
for $u=u^{+}+u^{-}\in E$, where $\Psi(u)=\displaystyle
 \int_{\mathbb{R}^{3}}F(x,u)dx$. Under assumptions $(f_1)$ and $(f_3)$, it follows by standard arguments that $I$ is well-defined
 and $\mathcal{C}^{2}$ on $E$. We say that $(\omega,u)\in \mathbb{R}\times E$ is a pair of (weak) normalized solution of problem \eqref{eq1.1} if for a fixed $a>0$ and any $\varphi\in \mathcal{C}_c^\infty (\mathbb{R}^3,\mathbb{C}^4)$,
 \[I'(u)\varphi - \omega\Re(u,\varphi)_{L^2}=0,\quad \|u\|_{L^2}=a,\]
 where $a>0$ is a prescribed constant.

\section{Existence of normalized solutions}
In this section, we discuss  the main steps of the proof and
self-contained approach that will enable us to prove Theorem
\ref{Thm1.1}. There are four major steps:\\

\underline{Step $1$}: Consider the following problem:
\begin{equation}\label{ca}
c_a=\inf_{v\in V\cap S_a}\sup_{u\in M(v)}I(u),
\end{equation}
where $a>0$, $v\in E^+$,  $S_a=\{u\in E\mid  \ \displaystyle
\int_{\mathbb{R}^{3}}|u|^{2}dx=a^2\}$, while $V$ and $M(v)$ will be
defined below.
\medskip

\underline{Step $2$}:  Motivated by Lyapunov-Schmidt reduction
method \cite{lyp}, we introduce a new functional $J$ defined on
$E^{+}$ by
$$c_a=\inf_{v\in V\cap S_a}J(v).$$

\underline{Step $3$}: Exploit the perturbation method on $J$ to
define a new operator $H$ defined as in \eqref{HH} that will ensure the existence
of the desired parameter $\omega_a$.\medskip

\underline{Step $4$}: Use the Ekeland's variational principle
\cite{ekeland} to conclude the proof of our result.\\

 We consider the following set
\begin{equation*}
    V:=\{v\in E^+\mid\|v\|<\sqrt{m+1}\|v\|_{L^2}\}.
\end{equation*}
It is clear that $V$ is an open subset of $E^+$ and the norms $\|\cdot\|$ and
$\|\cdot\|_{L^2}$ are equivalent on $V$.
For $v\in E^+$ with $v\neq0$, define
\begin{equation*}
M(v):=\left\{u\in E\mid \ \ \|u\|_{L^2}=\|v\|_{L^2}, \
u^+=\frac{\|u^+\|_{L^2}}{\|v\|_{L^2}}v\ \ \mbox{and} \ \ \|u^-\|\leq
\frac{\sqrt{m}\|v\|_{L^2}}{2} \right\}.
\end{equation*}

\begin{figure}[H]
\centering
\caption{Approximate image  of $M(v)$}
\includegraphics[scale=0.65]{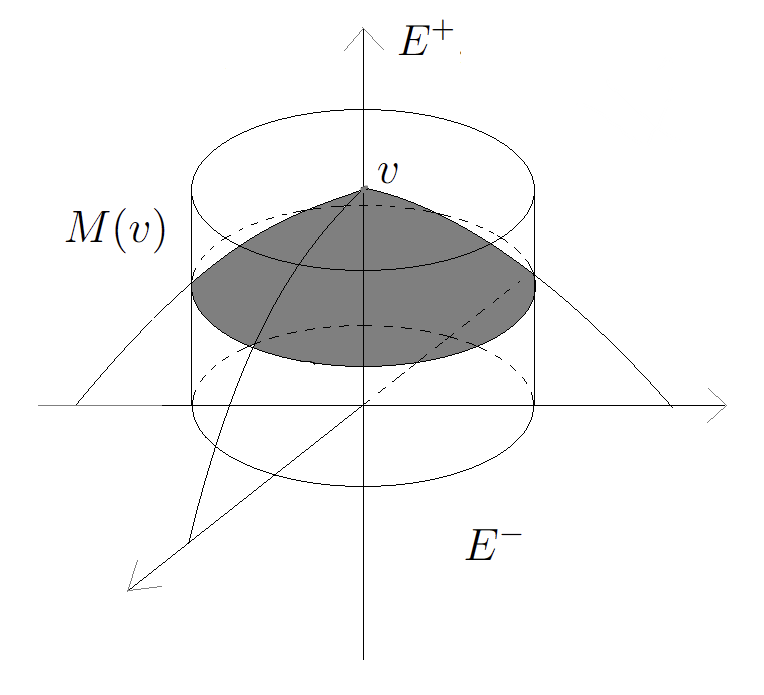}
\label{fig:1}
\end{figure}
 Next, we will demonstrate
that the supremum is indeed attained. Our goal is to provide a clear
and concise explanation and explore the properties and applications of $M(v)$.
 This step is essential for further discussion of the minimax value.
\begin{Lem}\label{Lem2.1}
Suppose that the assumptions of Theorem \ref{Thm1.1} are fulfilled.
Then, for any $v\in V$ with $\|v\|_{L^2}$ small enough, there exists
a unique $G(v)\in M(v)$ such that
\begin{equation*}
I(G(v))=\max_{u\in M(v)}I(u).
\end{equation*}
Moreover, $G$ is a continuously differentiable function.
\end{Lem}
\begin{proof}
    Let
    \begin{equation*}
        \mathscr{D}(h)=\left\{(v,w)\in V\times E^-\mid\|w\|\leq \frac{\sqrt{m}\|v\|_{L^2}}{2}\right\}
    \end{equation*}
and define the map $h:\mathscr{D}(h)\to E$ by
\begin{equation}\label{hv}
h(v,w)=\sqrt{\|v\|_{L^2}^2-\|w\|_{L^2}^2}\frac{v}{\|v\|_{L^2}}+w.
\end{equation}
Clearly, from \eqref{eq2.1}, we have $$\displaystyle\|w\|_{L^2}\leq
\frac{1}{\sqrt{m}}\|w\|\leq \frac{\|v\|_{L^2}}{2} \ \ \mbox{if} \ \
(v,w)\in \mathscr{D}(h),$$ and so $h$ is well-defined. Given $v\in
V$, we denote
\begin{equation*}
    \mathcal{B}(v)=\left\{w\in E^-\mid\|w\|\leq \frac{\sqrt{m}\|v\|_{L^2}}{2}\right\}.
\end{equation*}
The map $h(v,\cdot): \mathcal{B}(v)\rightarrow M(v)$ is a surjection, we can deduce that
\begin{equation}\label{eqq}
\sup_{u\in M(v)} I(u)=\sup_{u\in B(v)} I\circ h(v,u).
\end{equation}

 {\bf Claim $1$:} $(I\circ h)(v,\cdot)$ is
strictly concave on $\mathcal{B}(v)$ for $\|v\|_{L^2}$ small enough.

\medskip
\noindent Indeed, for any $z\in E^-$,
 one has
\begin{equation}\label{eq2.2}
\frac{\partial h}{\partial w}(v,w)z
=z-\frac{(w,z)_{L^2}}{\|v\|_{L^2}\sqrt{\|v\|_{L^2}^2-\|w\|_{L^2}^2}}v.
\end{equation}
Therefore, a direct computation in combination with \eqref{eq2.2}
gives
\begin{equation*}
    \begin{split}
    \left\langle \frac{\partial (I\circ h)}{\partial w}(v,w),z \right\rangle
        &=    \left\langle I^\prime(h(v,w)),\frac{\partial h}{\partial w}(v,w)z\right\rangle\\
        &
        =-(w,z)-\Re\int_{\mathbb{R}^3}f(x,|h(v,w)|)h(v,w)\cdot zdx\\
        &\quad-\Re(w,z)_{L^2}\left(\frac{\|v\|^2}{\|v\|_{L^2}^2}-
        \frac{\displaystyle\Re\int_{\mathbb{R}^3}
            f(x,|h(v,w)|)h(v,w)\cdot vdx}{\|v\|_{L^2}\sqrt{\|v\|_{L^2}^2-\|w\|_{L^2}^2}}\right),
    \end{split}
\end{equation*}
and
\begin{equation}\label{num1}
\frac{\partial^2(I\circ
    h)}{\partial w^2}(v,w)[z,z]=-\|z\|^2-\frac{\|v\|^2\|z\|_{L^2}^2}{\|v\|_{L^2}^2}+\sum_{j=1}^6I_j,
\end{equation}
where
\begin{equation*}
    I_1=-\Re \int_{\mathbb{R}^3}
f\left(x,|h(v,w)|\right)\left(|z|^2-\frac{\Re(w,z)_{L^2}}{\|v\|_{L^2}\sqrt{\|v\|_{L^2}^2-\|w\|_{L^2}^2}}v\cdot z\right)dx,
\end{equation*}

\begin{equation*}
    I_2=- \int_{\mathbb{R}^3}
\frac{f^\prime(x,|h(v,w)|)}{|h(v,w)|}\big(\Re \left(h(v,w)\cdot z\right)\big)^2dx,
\end{equation*}

\begin{equation*}
    I_3=\frac{2\Re(w,z)_{L^2}}{\|v\|_{L^2}\sqrt{\|v\|_{L^2}^2-\|w\|_{L^2}^2}}\int_{\mathbb{R}^3}
\frac{f^\prime\left(x,|h(v,w)|\right)}{|h(v,w)|}\Re\left(h(v,w)\cdot z\right)\Re\left(h(v,w)\cdot v\right)
dx,
\end{equation*}

\begin{align*}
    I_4&=\frac{\Re(w,z)_{L^2}}{\|v\|_{L^2}\sqrt{\|v\|_{L^2}^2-\|w\|_{L^2}^2}}\Re \int_{\mathbb{R}^3} f\left(x,|h(v,w)|\right)\left(v\cdot z-\frac{ (w,z)_{L^2}}{\|v\|_{L^2}
        \sqrt{\|v\|_{L^2}^2-\|w\|_{L^2}^2}}|v|^2\right)dx,\nonumber
\end{align*}

\begin{equation*}
    I_5=\frac{-\Re(w,z)_{L^2}^2}{\|v\|_{L^2}^2\left(\|v\|_{L^2}^2-\|w\|_{L^2}^2\right)}\int_{\mathbb{R}^3}
    \frac{f^\prime(x,|h(v,w)|)}{|h(v,w)|}\big(\Re \left(h(v,w)\cdot v\right)\big)^2
    dx,
\end{equation*}

and
\begin{equation*}
    I_6=\frac{\Re(w,z)_{L^2}^2}{\|v\|_{L^2}\left(\|v\|_{L^2}^2-\|w\|_{L^2}^2\right)^{3/2}}\Re \int_{\mathbb{R}^3} f\left(x,|h(v,w)|\right)h(v,w)\cdot vdx.
\end{equation*}
By the definition of $h$ and the equivalent of norms on $V$, we can check that
\begin{equation}\label{eq2.3}
\|h(v,w)\|\leq C\|v\|_{L^2}.
\end{equation}
Note that, by H\"older inequality, \eqref{eq1.2}, \eqref{eq1.5}, \eqref{eq2.3} and
for $\|v\|_{L^2}$ small enough, we get
\begin{align}\label{numm2}
    |I_1|&\leq C\int_{\mathbb{R}^3}\left(|h(v,w)|^{p-2}+|h(v,w)|^{q-2}\right)\left(|z|^2+\frac{\|z\|_{L^2}}{\|v\|_{L^2}}|v||z|\right)dx\\
    &\leq \frac{C\|z\|_{L^2}}{\|v\|_{L^2}}\left(\|h(v,w)\|_{L^p}^{p-2}\|z\|_{L^p}\|v\|_{L^p}+\|h(v,w)\|_{L^q}^{q-2}\|z\|_{L^q}\|v\|_{L^q}\right)\nonumber\\
    &\quad+C\left(\|h(v,w)\|_{L^p}^{p-2}\|z\|_{L^p}^2+\|h(v,w)\|_{L^q}^{q-2}\|z\|_{L^q}^2\right)\nonumber\\
    &\leq C\left(\|h(v,w)\|^{p-2}+\|h(v,w\|^{q-2}\right)\|z\|^2\nonumber\\
    &\leq C\left(\|v\|_{L^2}^{p-2}+\|v\|_{L^2}^{q-2}\right)\|z\|^2\nonumber\\
    &\leq \frac{1}{8}\|z\|^2\nonumber.
\end{align}
A similar estimation yields
\begin{equation}\label{num3}
        |I_j|\leq \frac{1}{8}\|z\|^2,~~j=2,3,\cdots,6.
\end{equation}
Consequently, due to \eqref{num1}, \eqref{numm2} and \eqref{num3},
we infer that
\begin{equation*}
\frac{\partial^2(I\circ
    h)}{\partial w^2}(v,w)[z,z]\leq -\frac{1}{4}\|z\|^2,
\end{equation*}
which means that $(I\circ h)(v,w)$ is strictly concave in $w$.

\medskip
\textbf{Claim $2$:} $(I\circ h)(v,\cdot)$ cannot
achieve its maximum on the boundary of $\mathcal{B}(v)$.
\medskip

\noindent If $v\in V$
with $\|v\|_{L^2}$ small enough and
$\displaystyle\|w\|=\frac{\sqrt{m}\|v\|_{L^2}}{2}$, we have
    \begin{equation*}
    \begin{split}
&\quad (I\circ h)(v,0)-(I\circ h)(v,w)\\
&
=\frac{\|w\|_{L^2}^2\|v\|^2}{2\|v\|_{L^2}^2}+\frac{1}{2}\|w\|^2-\int_{\mathbb{R}^3}F(x,|v|)dx+\int_{\mathbb{R}^3}F(x,|h(v,w)|)dx\\
&\geq \frac{m}{8}\|v\|_{L^2}^2-C\|v\|_{L^2}^p\\
&\geq \frac{m}{16}\|v\|_{L^2}^2.
    \end{split}
\end{equation*}
This proves the claim.
\medskip

 As a consequence of \eqref{eqq}, Claims $1$ and $2$, we conclude that there exists a unique
$w(v)\in \mathcal{B}(v)$ such that
\begin{equation*}
    I(h(v,w(v)))=\sup_{\tilde{w}\in \mathcal{B}(v)}I(h(v,\tilde{w})).
\end{equation*}
 By letting
\begin{equation}\label{G}
    G(v)=h(v,w(v)),
\end{equation}
then $G(v)\in M(v)$ satisfies
\begin{equation*}
    I(G(v))=\sup_{u\in M(v)}I(u).
\end{equation*}
Moreover, we have $\displaystyle \frac{\partial (I\circ h)}{\partial w}(v,w(v))=0$. Then, by the implicit function
theorem, we have $w(v)$ is continuously differentiable. This ends the proof .
\end{proof}
\subsection{The perturbed and reduced functional}
In this part, we construct a new functional related to the
functional $I$ defined on $E^{+}$ and contains a perturbed part that
will lay an important role to get the parameter $\omega_a$ in
equation \eqref{eq1.1}.\medskip

 We denote,
for $u\in E$ with $u^+\neq0$
    \begin{equation*}
\kappa(u)=\frac{\|u^+\|^2-\Re\displaystyle\int_{\mathbb{R}^3}f(x,|u|)u\cdot u^+dx}{\|u^+\|_{L^2}^2}
\end{equation*}
and consider the operator $H$ given by
    \begin{equation}\label{HH}
Hu:=H_0u-f(x,|u|)u-\kappa(u)u.
\end{equation}
Then by the definition of $G$ (see \eqref{G}), for any $z\in E^-$,
there holds that
\begin{equation}\label{eq2.4}
\Re\left(H\circ G(v),z\right)_{L^2}=0.
\end{equation}

For $a>0$ small enough, we consider the following functional $J:V\to
\mathbb{R}$ defined by
\begin{equation*}
    J(v)=(I \circ G)(v).
\end{equation*}
Invoking Lemma \ref{Lem2.1}, the functional $J$ is of class $\mathcal{C}^1$.
Let
\begin{equation*}
    c_a=\inf_{v\in V\cap S_a}J(v).
\end{equation*}
Notice that, for any $v \in V \cap S_a$, we have
$$
J(v)
=\frac{a^2-\|w(v)\|_{L^2}^2}{2a^2}\|v\|^2-\frac{1}{2}\|w(v)\|^2-\Psi(G(v))
\leq \frac{1}{2}\|v\|^2
$$
and so
\begin{equation}\label{numero}
    c_a=\inf _{v \in V \cap S_a}J(v) \leq \inf _{v \in V \cap S_a} \frac{1}{2}\|v\|^2=\frac{ m a^2}{2}.
\end{equation}
Furthermore, by constructing a special sequence, we will show that
$\displaystyle c_a<m a^2/2$, which  plays an important role in proving
our main theorems. For this purpose, we need the following Lemma.
\begin{Lem}\label{Lem2.2}
Under the assumptions of Theorem \ref{Thm1.1}, there exists a
bounded sequence $\{\nu_n\}\subset V$ such that
    \begin{equation*}
        \lim_{n\to \infty}\frac{\|\nu_n\|^2-m}{\Psi(\nu_n)}=0
    \end{equation*}
    and
    \begin{equation*}
        \|\nu_n\|_{L^2}=1,~~\Psi(\nu_n)>0,~n\in \mathbb{N}.
    \end{equation*}
\end{Lem}
\begin{proof}
 In light of Remark \ref{Rek1.1}, with slight modification, the proof is similar to the one in \cite[Lemma 2.6]{DingDirac}.
\end{proof}
\begin{Lem}\label{Lem2.3}
 Under the assumptions of Theorem \ref{Thm1.1}, if the prescribed constant $a$ small enough, then
 there holds $$c_a<\frac{ma^2}{2}.$$
\end{Lem}
\begin{proof}
    Let $\hat{\nu}_n=a\nu_n$, where $\nu_n$ is defined in Lemma \ref{Lem2.2}, then $\hat{\nu}_n\in V\cap S_a$, and so
    \begin{equation*}
        c_a\leq J(\hat{\nu}_n).
    \end{equation*}
Using the fact that
$$\displaystyle\frac{\sqrt{\|\hat{\nu}_n\|_{L^2}^2-\|w(\hat{\nu}_n)\|_{L^2}^2}}{2\|\hat{\nu}_n\|_{L^2}}
    \in (\frac{1}{4},\frac{1}{2}),$$
    together with \eqref{eq1.3}, \eqref{eq1.5}, \eqref{G} and the convexity of $\Psi$, we
    infer that
    \begin{equation*}
        \begin{split}
            J(\hat{\nu}_n)&=\frac{a^2-\|w(\hat{\nu}_n)\|_{L^2}^2}{2a^2}\|\hat{\nu}_n\|^2-\frac{1}{2}\|w(\hat{\nu}_n))\|^2-\Psi(G(\hat{\nu}_n))\\
            &\leq \frac{1}{2}\|\hat{\nu}_n\|^2-\frac{1}{2}\|w(\hat{\nu}_n))\|^2-2\Psi(\frac{\sqrt{a^2-\|w(\hat{\nu}_n)\|_2^2}}{2a}\hat{\nu}_n)+\Psi(-w(\hat{\nu}_n))\\
            &\leq \frac{1}{2}\|\hat{\nu}_n\|^2-\frac{1}{2}\|w(\hat{\nu}_n))\|^2-2^{1-2q}\Psi(\hat{\nu}_n)+\Psi(-w(\hat{\nu}_n))\\
            &\leq \frac{1}{2}\|\hat{\nu}_n\|^2-\frac{1}{2}\|w(\hat{\nu}_n)\|^2- 2^{1-2q}\Psi(\hat{\nu}_n)+C\|w(\hat{\nu}_n)\|^p+C\|w(\hat{\nu}_n)\|^q\\
            &\leq \frac{a^2}{2}\|\nu_n\|^2-2^{1-2q}a^q\Psi(\nu_n)-\frac{1}{2}\|w(\hat{\nu}_n)\|^2+C\|w(\hat{\nu}_n)\|^p+C\|w(\hat{\nu}_n)\|^q.
        \end{split}
    \end{equation*}
    Note that the boundedness of $\{\nu_n\}$ and the fact that $G\in \mathcal{C}^1$ imply that
    \begin{equation*}
        -\frac{1}{2}\|w(\hat{\nu}_n)\|^2+C\|w(\hat{\nu}_n)\|^p+C\|w(\hat{\nu}_n)\|^q\leq 0
\end{equation*}
    for $a$ small enough. Thus,
    \begin{equation*}
        J(\hat{\nu}_n)\leq  \frac{a^2}{2}\|\nu_n\|^2-2^{1-2q}a^q\Psi(\nu_n)
        <\frac{ma^2}{2}
    \end{equation*}
    if
    \begin{equation*}
        \frac{\|\nu_n\|^2-m}{\Psi(\nu_n)}<2^{2-2q}a^{q-2}
    \end{equation*}
    which holds for $n$ large enough in view of Lemma \ref{Lem2.2}. Thus, there holds
    \begin{equation*}
        c_a\leq J(\hat{\nu}_n)<\frac{ma^2}{2}.
    \end{equation*}
This completes the proof.
\end{proof}

\begin{Lem}\label{Lem2.5}
 Under the assumptions of Theorem \ref{Thm1.1}, for any $\mu>0$, there exists $c(\mu)>0$ such that, for any $a \in(0, c(\mu))$, if $v \in V \cap S_a$ satisfies
        $$
        J(v) \leq \frac{1}{2}\left(m+\frac{\mu}{2} a^{\frac{p-2}{2}}\right)
        a^2,
        $$
        then
        $$
        \|v\|^2\leq\left(m+\mu a^{\frac{p-2}{2}}\right) a^2.
        $$
\end{Lem}
\begin{proof}
By the definition of $J$ and $v\in M(v)$, we have
    \begin{equation*}
        J(v)  \geq I(v)  =\frac{1}{2}\|v\|^2-\Psi(v).
    \end{equation*}
    Thus, by Sobolev inequality and the equivalent of norms on $V$, one has
    \begin{equation*}
        \|v\|^2\leq\left(m+\frac{\mu}{2} a^{\frac{p-2}{2}}\right) a^2+Ca^p=\left(m+\frac{\mu}{2} a^{\frac{p-2}{2}}+Ca^{p-2}\right) a^2,
\end{equation*}
which holds for $a$ small enough. Let
\begin{equation*}
    c(\mu)=\left(\frac{\mu}{2 C}\right)^{\frac{2}{p-2}},
\end{equation*}
and the conclusion holds.
\end{proof}
\begin{Rek}\label{X}
Lemma \ref{Lem2.5} shows that, without loss of generality, each
minimizing sequence of $J$ restricted to $V \cap S_a$ remains
uniformly bounded away from the boundary of $V$. Thus, for $a>0$
small enough, if $v \in V \cap S_a$ and $J(v)$ is close to $c_a$, we
can assume that $v$ belongs to the following set
\begin{equation*}
    X_a=\left\{v\in E^+\mid\|v\|_{L^2}=a~\text{and}~\|v\|^2\leq \left(m+a^{\frac{p-2}{2}}\right)a^2\right\}\subset V \cap S_a.
\end{equation*}
Note that, $J_{|X_a}\in \mathcal{C}^1(X_a,\mathbb{R})$ and for all $v\in X_a$
\begin{equation*}
\left\langle J_{|X_a}^{\prime}\left(v\right),
z\right\rangle=\left\langle J^{\prime}\left(v\right),
z\right\rangle~~\text{for}~~z\in T_v,
\end{equation*}
where
\begin{equation*}
T_v=\left\{z\in E^+\mid\Re (v,z)_{L^2}=0\right\}.
\end{equation*}
\end{Rek}

\subsection{Proof of Theorem \ref{Thm1.1}}
In this subsection, we give the proof of Theorem \ref{Thm1.1}.
First, we recall the definition of the Palais-Smale condition which
is crucial in applying the critical point theory. We say that
$J_{\mid X_a}$ satisfies the Palais-Smale condition at a level $c
\in \mathbb{R},(P S)_c$ in short, if any sequence
$\left\{u_n\right\} \subset X_a$ with $J\left(u_n\right) \rightarrow
c$
 and $J_{\mid X_{a}}^{\prime}(u_n) \rightarrow 0$ contains a convergent subsequence.
 \begin{Lem}\label{Lem3.1}
Suppose that the assumptions of Theorem \ref{Thm1.1} are satisfied.
Then, for $a$ small enough, the functional $J$ satisfies $(PS)_c$
condition restricted to $X_a$ for $c<ma^2/2$.
 \end{Lem}
\begin{proof}
Let $\{v_n\}\subset X_a$ be any sequence such that
\begin{equation*}
J(v_n)\to c<\frac{ma^2}{2}~~\text{and}~~J_{|X_a}^\prime (v_n)\to
0~~\text{as}~n\to \infty.
\end{equation*}
Without loss of generality, we may assume that,
\begin{equation*}
    |\left\langle J ^{\prime}\left(v_n\right), z\right\rangle|\leq \frac{1}{n}\|z\|,~~\forall z \in T_{v_n}.
\end{equation*}
We claim that,
\begin{equation}\label{eq3.1}
    \langle J^\prime (v_n),z \rangle=\frac{\sqrt{a^2-\|w(v_n)\|_{L^2}^2}}{a}
   \Re \left(H\circ G(v_n),z\right)_{L^2},~~\forall z\in T_{v_n}.
\end{equation}
We first note that for any $z\in T_{v_n}$, we have
$$
\Re(G(v_n),G^\prime(v_n)z)_{L^2}=\frac{\sqrt{\|v_n\|_{L^{2}}^{2}-\|w(v_n)\|_{L^{2}}^{2}}}{\|v_n\|_{L^{2}}^{2}}\Re(v_n,z)_{L^2}=0.
$$
Thus
\begin{equation}\label{eq3.2}
    \begin{split}
    \langle J^\prime (v_n),z \rangle&=\langle I^\prime(G(v_n)),G^\prime (v_n)z \rangle-\kappa(G(v_n))\Re\left(G(v_n)),G^\prime (v_n)z\right)_{L^2}\\
    &=\Re\left(H\circ G(v_n),G^\prime (v_n)z\right)_{L^2}.
    \end{split}
\end{equation}
A direct computation yields
\begin{equation}\label{eq3.3}
    G^\prime(v_n)z=\frac{\sqrt{a^2-\|w(v_n)\|_{L^2}^2}}{a}z-\frac{\Re(w(v_n),w^\prime(v_n)z)_{L^2}}{a\sqrt{a^2-\|w(v_n)\|_{L^2}^2}}v_n+w^\prime(v_n)z
\end{equation}
and, by the definition of $h$, we have
\begin{equation}\label{eq3.4}
   \Re (H\circ G(v_n),v_n)_{L^2}=0.
\end{equation}
Thus, it follows from \eqref{eq3.2}-\eqref{eq3.4} and \eqref{eq2.4}
that
$$
\begin{aligned}
    \langle J^\prime (v_n),z \rangle & =\Re\left(H\circ G(v_n),(G^\prime (v_n)z)^+\right)_{L^2} =\frac{\sqrt{a^2-\|w(v_n)\|_{L^2}^2}}{a}
    \Re\left(H\circ G(v_n),z\right)_{L^2}
\end{aligned}
$$
and then claim \eqref{eq3.1} holds.

\medskip
For any $z\in E^+$, it is easy to see that
\begin{equation*}
    z-\frac{\Re(v_n,z)_{L^2}}{a^2}v_n\in T_{v_n}.
\end{equation*}
Therefore, in view of \eqref{eq3.1}, we obtain
$$
\begin{aligned}
    \Re\left(H\circ G(v_n),z\right)_{L^2}& =   \Re\left(H\circ G(v_n),z-\frac{\Re(v_n,z)_{L^2}}{a^2}v_n\right)_{L^2} \\
    &=\frac{a}{\sqrt{a^2-\|w(v_n)\|_{L^2}^2}}\left\langle J^\prime (v_n),z-\frac{\Re(v_n,z)_{L^2}}{a^2}v_n \right\rangle\\
    &\leq \frac{a}{n\sqrt{a^2-\|w(v_n)\|_{L^2}^2}}\left\|z-\frac{\Re(v_n,z)_{L^2}}{a^2}v_n \right\|\\
    &\leq \frac{a}{n\sqrt{a^2-\|w(v_n)\|_{L^2}^2}}\left(1+\sqrt{\frac{m+1}{m}}\right)\|z\|\\
    &\leq \frac{C}{n}\|z\|,
\end{aligned}
$$
and so, from \eqref{eq2.4} again, we get
\begin{equation}\label{H}
    H\circ G(v_n)\to 0~\text{in}~H^{-1/2}\left(\mathbb{R}^3,\mathbb{C}^4\right).
\end{equation}
Up to a subsequence, we may assume that
\begin{equation*}
    G(v_n)\rightharpoonup u_a ~~~\text{in}~~~ E ~~~\text{and}~~~\kappa(G(v_n))\to \omega_a \
    \ \mbox{in} \ \ \mathbb{R}
\end{equation*}
and then $v_n$ and $u_a$ satisfy the following equations, respectively
\begin{equation}\label{eq3.6}
    -i\alpha\cdot \nabla G(v_n)+m\beta G(v_n)=f(x,|G(v_n)|)G(v_n)+\kappa(G(v_n))G(v_n)+o_n(1),
\end{equation}
\begin{equation}\label{eq3.7}
    -i\alpha\cdot \nabla u_a+m\beta u_a=f(x,|u_a|)u_a+\omega_a u_a.
\end{equation}
The next goal is to show that $G(v_n)\to u_a$ in $E$. Note that, in
view of \eqref{eq1.2} and \eqref{H}, we obtain
$$
\begin{aligned}
    \omega_a&=\lim_{n\to \infty}\kappa(G(v_n))\\
    &=\frac{1}{a^2}\lim_{n\to \infty}\left[\Re(H_0G(v_n),G(v_n))_{L^2}-\int_{\mathbb{R}^3}f(x,|G(v_n)|)|G(v_n)|^2dx\right]\\
    & =\frac{1}{c^2}\lim_{n\to \infty}\left[2J(v_n)+2\int_{\mathbb{R}^3}F(x,|G(v_n)|)dx-\int_{\mathbb{R}^3}f(x,|G(v_n)|)|G(v_n)|^2dx\right]\\
    & \leq \frac{2}{a^2}\lim_{n\to \infty}J(v_n)\\
    & =\frac{2c}{a^2}<m,
\end{aligned}
$$
The scalar product of \eqref{eq3.6} and \eqref{eq3.7} with
$(G(v_{n})-u_{a})^{+}$, respectively, we get
\begin{equation}\label{eq3.8}
    \begin{split}
        \|(G(v_n)-u_a)^+\|^2=&\Re\int_{\mathbb{R}^3}(f(x,|G(v_n)|)G(v_n)-f(x,|u_a|)u_a)\cdot (G(v_n)-u_a)^+dx\\
        &+\Re\left(\kappa(G(v_n))G(v_n)-\omega_a u_a,(G(v_n)-u_a)^+\right)_{L^2}+o_n(1).
    \end{split}
\end{equation}
Note that, from the assumption $(f_4)$, we can check
that
\begin{equation}\label{eq3.9}
    \Re\int_{\mathbb{R}^3}(f(x,|G(v_n)|)G(v_n)-f(x,|u_a|)u_a)\cdot(G(v_n)-u_a)^+dx\to 0
\end{equation}
as $n\to \infty$. Moreover, we have
\begin{equation}\label{eq3.10}
    \begin{split}
        &\quad  \left|\Re\left(\kappa(G(v_n))G(v_n)-\omega_a u_a,(G(v_n)-u_a)^+\right)_{L^2}\right|\\
        &=\left|(\kappa(G(v_n))-\omega_a)\Re(G(v_n),(G(v_n)-u_a)^+)_{L^2}+ \omega_a\|(G(v_n)-u_a)^+\|_{L^2}^2\right|\\
        &\leq o_n(1)\|(G(v_n)-u_a)^+\|_{L^2}+\omega_a\|(G(v_n)-u_a)^+\|_{L^2}^2.
    \end{split}
\end{equation}
Thus, it follows from \eqref{eq3.8}-\eqref{eq3.10} that $G^+(v_n)\to
u_a^+$ in $E$. Similarity, we can obtain $G^-(v_n)\to u_a^-$ in $E$
and then $G(v_n)\to u_a$ in $E$. By the definition of $G$ and $v_n$, we have $v_n\rightarrow v_a$ in $E$, where $v_a$ belongs to $X_a$.
\end{proof}
\medskip

\noindent {\bf The proof of Theorem 1.1. (Existence)} We divide the proof into
two steps:\medskip

\textbf{Step $1$:} Existence of $(PS)_{c_{a}}$-sequence of $J$ restricts to
$X_a$.

\medskip
Let $\{v_n\}\subset V\cap S_a$ be a minimizing sequence
with respect to the $J$ such that
\begin{equation*}
    J(v_n)\to c_a.
\end{equation*}
 Then, by Lemmas \ref{Lem2.3} and
\ref{Lem2.5}, for $c>0$ small enough, up to a subsequence, we can
assume that
\begin{equation*}
    \|v_n\|^2\leq\left(m+\frac{1}{2} a^{\frac{p-2}{2}}\right) a^2,~\forall n \in \mathbb{N}.
\end{equation*}
By the Ekeland variational principle on the complete metric space
$X_a$, there exists $v_n^*\in X_a$ such that
\begin{itemize}
    \item[$(1)$] $J(v_n^*)\leq J(v_n)$.
    \item[$(2)$] $\|v_n-v_n^*\|\leq \displaystyle\frac{1}{n}$.
    \item[$(3)$] $J(v_n^*)\leq J(v)+\displaystyle\frac{1}{n}\|v-v_n^*\|$ for all $v\in X_a$.
\end{itemize}
Clearly, $J\left(v_{n}^*\right) \rightarrow c_a$ and we can assume
that $v_n^*\rightharpoonup v_a$ in $E^+$ and satisfies
$$
\|v_n^*\|^2\leq\left(m+\frac{3}{4} a^{\frac{p-2}{2}}\right) a^2.
$$
Given $z\in T_{v_n^*}=\left\{z\in E^+\mid\Re(v_n^*,z)_{L^2}=0\right\}$, there is
$\delta>0$ small enough such that the path $\gamma:(-\delta,
\delta)\longrightarrow S_a \cap E^+$ defined by
$$
\gamma(t)=\sqrt{1-t^2 \frac{\|z\|_{L^2}^2}{a^2}} v_n^*+t z,
$$
belongs to $\mathcal{C}^1\left((-\delta,\delta),S_a\cap E^+\right)$ and
satisfies
\begin{equation*}
    \gamma(0)=v_n^*~~~\text{and}~~~\gamma^\prime(0)=z.
\end{equation*}
Hence,
\begin{equation*}
    J(\gamma(t))-J(v_n^*)\geq -\frac{1}{n}\left\|\gamma(t)-v_n^*\right\|,
\end{equation*}
and in particular,
$$
\frac{J(\gamma(t))-J(\gamma(0))}{t}=\frac{J(\gamma(t))-J(v_n^*)}{t}\geq
-\frac{1}{n}\left\|\frac{\gamma(t)-v_n^*}{t}\right\|=-\frac{1}{n}\left\|\frac{\gamma(t)-\gamma(0)}{t}\right\|,~~\forall
t\in (0,\delta).
$$
Since $J\in \mathcal{C}^1$, taking the limit of $t\to 0^+$, we get
\begin{equation*}
    \left\langle J^{\prime}\left(v_n^*\right), z\right\rangle\geq -\frac{1}{n}\|z\|.
\end{equation*}
Replacing $z$ by $-z$, we obtain
\begin{equation*}
    |\left\langle J ^{\prime}\left(v_n^*\right), z\right\rangle|\leq \frac{1}{n}\|z\|,~~\forall z
    \in
    T_{v_n^*}.
\end{equation*}
It follows, from Remark \ref{X}, that $\{v_n^*\}$ is a $(PS)_{c_a}$-sequence for $J$ restricted to $X_a$.\medskip

\textbf{Step $2$}: The existence of solutions for problem
\eqref{eq1.1}.\medskip

 Since $c_a<\displaystyle\frac{ma^2}{2}$ by Lemma \ref{Lem2.3},
so, Step $1$ and Lemma \ref{Lem3.1}  ensure that there is $u_a$ such
that $G(v_n^*) \rightarrow u_a$ in
$H^{1/2}\left(\mathbb{R}^3,\mathbb{C}^4\right)$ and
$k(G(v_n^*))\rightarrow \omega_a$ in $\mathbb{R}$. Therefore,
arguing as in the proof of Lemma \ref{Lem3.1} and using \eqref{H},
we can conclude that
\begin{equation*}
    -i\alpha\cdot \nabla u_a+m\beta u_a=f(x,|u_a|)u_a+\omega_a u_a,
\end{equation*}
and
\begin{equation*}
\|u_a\|_{L^2}=\lim_{n\to \infty}\|G(v_n^*)\|_{L^2}=\lim_{n\to
\infty}\|v_n^*\|_{L^2}=a,
\end{equation*}
and the existence is complete.

\medskip

{\bf (Bifurcation)} Note that, from the proof of Lemma \ref{Lem3.1},
we have
\begin{equation*}
 \omega_a\leq \frac{2c_a}{a^2}<m.
\end{equation*}
Moreover, by the definition of $h$  (see \eqref{hv}), for $a$ small enough
\begin{equation*}
\kappa(G(v_n^*))=\frac{\|G^+(v_n^*)\|^2}{\|G^+(v_n^*)\|_{L^2}^2}-\frac{\Re
\displaystyle\int_{\mathbb{R}^3}f(x,|G(v_n^*)|)G(v_n^*)\cdot G^+(v_n^*)dx}{\|G^+(v_n^*)\|_{L^2}^2}\geq
m-Ca^{p-2}.
\end{equation*}
Thus,
\begin{equation*}
    m-Ca^{p-2}\leq \omega_a\leq \frac{2c_a}{a^2}<m,
\end{equation*}
which implies that
\begin{equation*}
\omega_a\to m~~\text{as}~~a\to 0^+.
\end{equation*}
On the other hand,
\begin{equation*}
    \|u_a\|= \lim_{n\to \infty}\|G(v_n^*)\|\leq \lim_{n\to \infty} \sqrt{\|v_n^*\|^2+\|w(v_n^*)\|^2}\leq \frac{\sqrt{5m+4}}{2}a\to 0~~\text{as}~~a\to 0^+.
\end{equation*}
Thus, $m$ is a bifurcation point on the left of equation
\eqref{eq1.1}.

\section{Multiplicity of Normalized solutions}

In this section, we always assume that $a>0$ is fixed and we set
$X=X_a$. To get multiple normalized solutions, we will use some
arguments used in the last section in combination with the
multiplicity theorem of the Ljusternik-Schnirelmann type for
$J_{|X}$ established in \cite{jeanjean1992}. First, let us recall
the definition of genus. Let $\Sigma(X)$ be the family of closed
symmetric subsets of $X$. For any nonempty set $A \in \Sigma(X)$,
 let $\gamma(A)$ denote the Kransnoselskii genus of $A$. For the definition and properties of $\gamma(A)$, we refer to
 \cite{willem}, and we list here some of them without proofs.

\begin{Lem}
Let $A, B \in \Sigma(X)$, then the following statements hold:
    \begin{itemize}
    \item[$(i)$] if $\gamma(A) \geq 2$, then $A$ contains infinitely many distinct points;
    \item[$(ii)$] if $A \subset B$, then $\gamma(A) \leq \gamma(B)$;
    \item[$(iii)$] let $\Omega \subset \mathbb{R}^k$ be an open bounded symmetric neighborhood of 0, then $\gamma(\partial \Omega)=k$;
    \item[$(iv)$] if there exists an odd continuous mapping $\psi: \mathbb{S}^{k-1} \rightarrow A$, then $\gamma(A) \geq k$.
\end{itemize}
\end{Lem}
For each $k \in \mathbb{N}$, we define
$$
\Gamma_k:=\left\{A \in \Sigma(X)\mid \gamma(A) \geq k\right\}.
$$

\setcounter{Thm}{0}
\renewcommand{\theThm}{\Alph{Thm}}
\begin{Thm}
Assume that the functional $J$ satisfies $(PS)_c$ condition restricted to
$X$ for $c<\frac{1}{2}ma^2$ and set
    \begin{equation*}
        b_k=\inf_{A\in \Gamma_k}\sup_{v\in A}J(v),~~k\geq 1.
    \end{equation*}
    Then
    \begin{itemize}
        \item[$(1)$] $b_k$ is a critical value of $J_{|X}$ provided  that $\displaystyle b_k<\frac{ma^2}{2}$.
        \item[$(2)$] denoting by $K_b$ the set of critical points of $J_{\mid X}$ at a level $b\in \mathbb{R}$, if
        $$
        b_k=b_{k+1}=\cdots=b_{k+r-1}=:b<\frac{ma^2}{2}
        $$ for some $k,r\geq 1$,
        then $\gamma(K_b)\geq r$. In particular, $J_{\mid X}$ has infinitely many critical points at the level $b$ if $r \geq 2$.
        \item[$(3)$] if $b_k<\displaystyle\frac{ma^2}{2}$ for all $k\geq 1$, then there exist infinitely many critical points $\{v_a^n\}$ for $J_{\mid X}$ and these critical points satisfy $J(v_a^n)<\displaystyle\frac{ma^2}{2}$.
    \end{itemize}
\end{Thm}

\begin{Lem}\label{Lem4.1}(\cite[Lemma 3.6]{Heinz1993})
    Let $Z$ be a finite dimensional, $Y$ an arbirtary normed space, and let $\{L_k\}$ be a sequence of linear maps $Z\rightarrow Y$ for which
    \begin{equation*}
        \lim_{k\to \infty}\|L_kz\|=\|z\|,~\forall z\in Z.
    \end{equation*}
    Then
    \begin{itemize}
        \item[$(i)$] the convergence $\|L_kz\|\rightarrow \|z\|$ is uniform on bounded subsets of $Z$;
        \item[$(ii)$] there is $k_0\in \mathbb{N}$ such that $L_k$ is injective for $k\geq k_0$.
    \end{itemize}
\end{Lem}


 Given $k \in \mathbb{N}$, we choose a $k$-dimensional subspace $Z$, which is spanned of the first $k$ Hermite functions in $\mathbb{R}^3$, and we endow $Z$ with the $L^2$-norm.\medskip

  Now, we recall the following important lemma proved in \cite[Lemma 2.2]{DingDirac}.
\begin{Lem}\label{spe}
For any $\lambda\in \sigma(H_0)$, there exists a non-trivial
periodic solution $\Phi_{\lambda}\in
\mathcal{C}^{\infty}(\mathbb{R}^{3},\mathbb{C}^{4})$ of the following equation
  $$-i\sum\limits_{k=1}^3\alpha_k\partial_k u+m\beta u=\lambda u.$$
\end{Lem}

Next, we define the linear maps $L_k: Z \rightarrow
H^{1}(\mathbb{R}^{3},\mathbb{C}^{4})$ by
 \begin{equation}\label{eq4.1}
   \left(L_k \zeta\right)(x):=\frac{1}{k^{3/2}\sqrt{M\left(\left|\Phi_{m}\right|^{2}\right)}}\Phi_m(x) \zeta\left(\frac{x}{k}\right)
 \end{equation}
 for $\zeta \in Z, k \in \mathbb{N}, x \in \mathbb{R}^3$, while $\Phi_m$ is defined in Lemma \ref{spe} and $M(f)$ denotes the mean-value of a uniformly almost periodic function
 $f:\mathbb{R}^{3}\rightarrow \mathbb{R}$, which is defined by
$$M(f)=\displaystyle \lim_{T\rightarrow \infty}\frac{1}{T^{3}}\int_{0}^{T}\int_{0}^{T}\int_{0}^{T}f(x)dx.$$
  In what follows, we denote $Z^1:=\{\zeta \in Z \mid\|\zeta\|_{L^2}=1\}$.

\begin{Lem}\label{Lem4.2}
    Assume that $(f_1)$-$(f_5)$ hold and let $\{L_k\}$ be defined by \eqref{eq4.1}. Then, there exist constants $C_0>0$ and $k_0\in \mathbb{N}$ such that
    \begin{equation*}
        \Psi\left(L_k \zeta\right) \geq C_0 k^{-\frac{3(\alpha-2)}{2}-\tau},~~\forall k>k_0~~\text{and}~~\zeta\in Z^1.
    \end{equation*}
\end{Lem}
\begin{proof}
Since $\Phi_m$ is bounded and the sup-norm on $Z$ is equivalent to
the $L^2$-norm, there is $c_1>0$ such that
\begin{equation*}
\left|\left(L_k \zeta\right)(x)\right| \leq c_1
k^{-3/2}~~\text{for}~~k \in \mathbb{N}, ~~\zeta \in Z^1,~~ x \in
\mathbb{R}^3.
\end{equation*}
 Hence, there is $k_0$ such that
 \begin{equation*}
\left\|L_k \zeta\right\|_{\infty} \leq t_0~~\text{for}~~k \geq k_0
~~ \mbox{and}~~ \zeta \in Z^1,
 \end{equation*}
where $t_0$ is defined in $(f_5)$. By the definition of
$L_k$ and $(f_5)$, we have
    $$
    \begin{aligned}
        \Psi(L_k \zeta) & =\int_{ \mathbb{R}^3}F\left(x,|L_k \zeta|\right)dx\\
        & \geq \frac{Lk^{-\frac{3\alpha}{2}}}{
            M^{\frac{\alpha}{2}}\left(|\Phi_m|^2\right)}\int_S |x|^{-\tau}\left|  \Phi_m(x) \zeta\left(\frac{x}{k}\right)\right|^{\alpha}dx \\
        & =\frac{Lk^{-\frac{3\alpha}{2}+3-\tau}}{
            M^{\frac{\alpha}{2}}\left(|\Phi_m|^2\right)} \int_{S_k}|y|^{-\tau}|\Phi_m(k y)|^\alpha|\zeta(y)|^\alpha dy,
    \end{aligned}
    $$
    where $S_k:=\{y\in \mathbb{R}^3|ky\in S\}$. Note that, for $k\geq 1$, $S\subset S_k$. It follows, in the light of \cite[proposition 3.1]{Heinz-Kupper-Stuart1992JDE}, that
\begin{align}\label{eq4.2}
\frac{M^{\frac{\alpha}{2}}\left(|\Phi_m|^2\right)}{
    L}
k^{\frac{3(\alpha-2)}{2}+\tau} \Psi(L_k \zeta)&\geq
\int_{\mathbb{R}^3}\left|\Phi_m(k y)\right|^\alpha
\chi_S(y)|y|^{-\tau} |\zeta(y)|^\alpha dy\\ &\to
M\left(|\Phi_m|^\alpha\right) \int_{\mathbb{R}^3} g(y) dy,\nonumber
\end{align}
where
    $$
g(y):=\chi_S(y)|y|^{-\tau}|\zeta(y)|^\alpha,
$$
and $\chi_S$ is the characteristic function of the set $S$.

\medskip

Now apply Lemma \ref{Lem4.1} to the linear maps $\Lambda_k: Z \rightarrow L^\alpha\left(S ;|y|^{-\tau} \mathrm{d} y\right)$ given by
    $$
    \left(\Lambda_k \zeta\right)(x):=\Phi_m(k x) \zeta(x) \quad(x \in S)
    $$
    where the norm on $Z$ is now given by
    $$
    \|\zeta\|_*:=\left(M\left(|\Phi_m|^{\alpha}\right) \int_S|y|^{-\tau}|\zeta(y)|^\alpha dy\right)^{\frac{1}{\alpha}} .
    $$
    Note that, every $\zeta \in Z$ is real analytic, we may assume that for any $\zeta\in Z^1$
    \begin{equation*}
\zeta(x)\neq 0~\text{a.e.}~\text{on}~S.
    \end{equation*}
Therefore, it is easy to see that $\|\cdot\|_*$ is indeed a norm and
then by the equivalence of the norms on $Z$, there exists a constant $c_2>0$
such that
    \begin{equation}\label{eq44.3}
\|\zeta\|_*\geq c_2,~~ \forall \zeta \in Z^1.
    \end{equation}
Thus, we obtain from Lemma \ref{Lem4.1}, \eqref{eq4.2} and
\eqref{eq44.3}
    $$      \Psi\left(L_k \zeta\right) \geq \frac{Lc_2^\alpha}{
        M^{\frac{\alpha}{2}}\left(|\Phi_m|^2\right)}k^{-\frac{3(\alpha-2)}{2}-\tau} =: C_0k^{-\frac{3(\alpha-2)}{2}-\tau}>0,
    $$
    for $k>k_0$, $\zeta\in Z^1$. This completes the proof.
\end{proof}

\begin{Lem}\label{Lem4.3}
    For any $k\in \mathbb{N}$ and let $\{L_k\}$ be defined by \eqref{eq4.1}, then
    \begin{itemize}
        \item[$(i)$] $\lim\limits _{k \rightarrow \infty}\left\|L_k \zeta\right\|_{L^2}=\|\zeta\|_{L^2}$ for all $\zeta \in Z$;
        \item[$(ii)$]  there exists constant $C$ such that, for $k$ large enough
        \begin{equation*}
        0<\Psi(L_k \zeta) \leq C~ \text { for all } \zeta \in Z^1;
        \end{equation*}
        \item[$(iii)$]
        \begin{equation*}
            \frac{\sup\limits_{\zeta\in Z^1}
                 \left|((H_0-m) L_k \zeta, L_k \zeta)_{L^2}\right|}
            {\inf\limits_{\zeta\in Z^1} \Psi\left(L_k \zeta\right)} \rightarrow 0 \text { as } k \rightarrow \infty,
        \end{equation*}
    and
            \begin{equation*}
        \frac{\sup\limits_{\zeta\in Z^1}\|(H_0-m) L_k \zeta\|_{L^2}}
        {\inf\limits_{\zeta\in Z^1} \Psi\left(L_k \zeta\right)} \rightarrow 0 \text { as } k \rightarrow \infty .
    \end{equation*}
    \end{itemize}
\end{Lem}
\begin{proof}
$(i)$ Clearly, it follows from \cite[proposition
3.1]{Heinz-Kupper-Stuart1992JDE} that
\begin{equation*}
    \int_{\mathbb{R}^3}|L_k \zeta|^2dx=\frac{1}{M\left(|\Phi_m|^2\right)}\int_{\mathbb{R}^3}\zeta^2(x)|\Phi_m(kx)|^2dx\rightarrow \int_{\mathbb{R}^3}|\zeta|^2dx~\text{as}~k\to \infty.
\end{equation*}
$(ii)$ In view of Lemma \ref{Lem4.2}, we get $\Psi(L_k \zeta)>0$ for all $\zeta\in Z^1$. Thus, it suffices to show that $L_k(Z^1)$ is uniformly bounded in
$H^{1/2}(\mathbb{R}^3,\mathbb{C}^4)$. For any $\zeta \in
Z^1$, we have
\begin{equation}\label{eq4.3}
    \begin{split}
    \int_{\mathbb{R}^3}|\nabla \left(L_k \zeta\right)|^2dx&\leq \frac{2}{M(|\Phi_m|^2)k^2}\int_{\mathbb{R}^3}|\nabla\zeta(x)|^2|\Phi_m(kx)|^2dx\\
    &\quad+\frac{2}{M(|\Phi_m|^2)}\int_{\mathbb{R}^3}|\zeta(x)|^2|\nabla \Phi_m(kx)|^2dx.
    \end{split}
\end{equation}
By \cite[proposition 3.1]{Heinz-Kupper-Stuart1992JDE} again, there
holds
\begin{equation}\label{eq4.4}
\int_{\mathbb{R}^3}|\nabla \zeta(x)|^2|\Phi_m(kx)|^2dx\to
M(|\Phi_m|^2)\int_{\mathbb{R}^3}|\nabla \zeta|^2dx
\end{equation}
and
\begin{equation}\label{eq4.5}
    \int_{\mathbb{R}^3}|\zeta(x)|^2|\nabla \Phi_m(kx)|^2dx\rightarrow M(|\nabla \Phi_m|^2)\int_{\mathbb{R}^3}|\zeta|^2dx=M(|\nabla \Phi_m|^2).
\end{equation}
Since all norms on $Z$ are equivalent, there exists constant $c_3>0$
such that
\begin{equation}\label{eq4.7}
\int_{\mathbb{R}^3}|\nabla\zeta|^2dx\leq c_3.
\end{equation}
Note that,
\begin{equation*}
    \int_{\mathbb{R}^3}|L_k \zeta|^2dx\leq \|L_k \zeta\|^2\leq \|L_k \zeta\|_{H^1}^2=\int_{\mathbb{R}^3}|L_k \zeta|^2dx+\int_{\mathbb{R}^3}|\nabla (L_k \zeta)|^2dx.
\end{equation*}
Therefore, this together with \eqref{eq4.4}-\eqref{eq4.7}, there exist
constants $A,B>0$ such that, for $k$ large enough
\begin{equation*}
    0<A\leq \|L_k \zeta\|\leq B, \ \ \forall \zeta\in Z^1.
\end{equation*}
$(iii)$ For any $\zeta \in Z^1$, by a direct computation, we
obtain
\begin{equation*}
    (H_0 -mI)L_k \zeta(x)=\frac{-ik^{-\frac{5}{2}}}{M^{\frac{1}{2}}\left(|\Phi_m|^2\right)}\sum\limits_{j=1}^3\frac{\partial \zeta}{\partial x_j}(\frac{x}{k})\alpha_j \Phi_m(x)
\end{equation*}
since $(H_0 -mI)\Phi_m(x)=0$. Hence, it yields
\begin{equation}\label{eq4.8}
    ((H_0-mI)L_k \zeta,L_k \zeta)_{L^2}=\frac{-ik^{-4}}{M\left(|\Phi_m|^2\right)}\int_{\mathbb{R}^3}\sum_{j=1}^3\frac{\partial \zeta}{\partial x_j}(\frac{x}{k})\alpha_j \Phi_m(x)\zeta(\frac{x}{k})\Phi_m(x)dx.
\end{equation}
Using \eqref{eq4.8} and the Cauchy inequality, one has
\begin{equation}\label{eq4.9}
    \begin{split}
        &\quad \quad k|((H_0-mI)L_k \zeta,L_k \zeta)_{L^2}|\\
        &=\frac{k^{-3}}{M\left(|\Phi_m|^2\right)}\bigg|\int_{\mathbb{R}^3}\sum_{j=1}^3\frac{\partial \zeta}{\partial x_j}(\frac{x}{k})
        \alpha_j\Phi_m(x)\zeta(\frac{x}{k})\Phi_m(x)dx\bigg|\\
        &\leq \frac{1}{M\left(|\Phi_m|^2\right)}\int_{\mathbb{R}^3}|\sum_{j=1}^3\frac{\partial \zeta}{\partial x_j}(x)|\zeta(x)|\Phi_m(kx)|^2dx\\
        &\leq \frac{\sqrt{3}}{M\left(|\Phi_m|^2\right)}\Big[\int_{\mathbb{R}^3}|\nabla \zeta(x)|^2|\Phi_m(kx)|^2dx\Big]^{\frac{1}{2}}\Big[\int_{\mathbb{R}^3}|\zeta(x)|^2|\Phi_m(kx)|^2dx\Big]^{\frac{1}{2}}\\
        &\to \sqrt{3}\Big[\int_{\mathbb{R}^3}|\nabla \zeta|^2dx\Big]^{\frac{1}{2}}~\text{as}~k\to \infty
     \end{split}
\end{equation}
and
\begin{equation}\label{eq4.10}
    \begin{split}
        k^2\|(H_0-mI)L_k \zeta\|_{L^2}^2&=\frac{k^{-3}}{M\left(|\Phi_m|^2\right)}\int_{\mathbb{R}^3}|\sum_{j=1}^3\frac{\partial \zeta}{\partial x_j}(\frac{x}{k})|^2|\Phi_m(x)|^2dx\\
        &\leq \frac{3}{M\left(|\Phi_m|^2\right)}\int_{\mathbb{R}^3}|\nabla \zeta(x)|^2|\Phi_m(kx)|^2dx\\
        &\to 3\int_{\mathbb{R}^3}|\nabla \zeta|^2dx~\text{as}~k\to \infty.
    \end{split}
\end{equation}
Thus, from \eqref{eq4.7} and \eqref{eq4.9}-\eqref{eq4.10}, there
exists positive constant $C$ independent of $\zeta$ such that
\begin{equation}\label{eq2.10}
    |((H_0-mI)L_k \zeta,L_k \zeta)_{L^2}|\leq \frac{C}{k}
\end{equation}
and
\begin{equation}\label{eq2.11}
    \|(H_0-mI)L_k \zeta\|_{L^2}\leq \frac{C}{k}.
\end{equation}
It follows, from Lemma \ref{Lem4.2}, \eqref{eq2.10}-\eqref{eq2.11},
that
\begin{equation*}
    \frac{\sup\limits_{\zeta\in Z^1}\|(H_0-mI)L_k \zeta\|_{L^2}}{\inf\limits_{\zeta\in Z^1}\Psi(L_k \zeta)}\leq Ck^{\frac{3(\alpha-2)}{2}+\tau-1}\to 0
\end{equation*}
and
\begin{equation*}
    \frac{\sup\limits_{\zeta\in Z^1}\left|((H_0-mI)L_k \zeta,L_k \zeta)_{L^2}\right|}{\inf\limits_{\zeta\in Z^1}\Psi(L_k \zeta)}\leq Ck^{\frac{3(\alpha-2)}{2}+\tau-1}\to 0
\end{equation*}
as $k\to \infty$ since $\tau\in (0,\frac{8-3\alpha}{2})$, which
completes the proof of Lemma \ref{Lem4.3}.
\end{proof}
\begin{Lem}\label{Lem4.4}
Suppose that $(f_1)$-$(f_5)$ are satisfied. For any $k\in \mathbb{N}$,
there exists a sequence $\left\{X_n^k\right\}_{n=1}^\infty$ of
$k$-dimensional linear subspaces of $E^+$ such that for the sets
    $$
S_n^k:=\left\{u \in X_n^k \mid\|u\|_{L^2}=1\right\}
    $$
    we have
    $$
\frac{\sup\limits_{v\in S_n^k}
\left(\|v\|^2-m\right)}{\inf\limits_{v\in S_n^k}\Psi(v)} \rightarrow
0 \quad \text { as } n \rightarrow \infty .
    $$
\end{Lem}
\begin{proof}
    Fix $\zeta \in Z^1:=\left\{z \in Z\mid\|z\|_{L^2}=1\}\right.$ and let $u_n:=L_n\zeta$. Then, from
    \eqref{eq2.11}, one has
            \begin{equation*}
        \|(H_0-mI) u_n\|_{L^2}\to 0~\text{as}~n\to \infty.
    \end{equation*}
Thus, by Cauchy inequality, one has
            \begin{equation*}
    \|(H_0-mI) u_n\|_{L^2}\|u_n^-\|_{L^2}\geq |((H_0-mI) u_n,u_n^-)_{L^2}|=\|u_n^-\|^2+m\|u_n^-\|_{L^2}^2
\end{equation*}
which implies that
            \begin{equation}\label{eq04.13}
    \|u_n^-\|_{L^2}\to 0~\text{and}~\|u_n^-\|\to 0~\text{as}~n\to \infty,
\end{equation}
and
            \begin{equation}\label{eq4.13}
\|u_n^-\|_{L^2}\leq \|(H_0-mI) u_n\|_{L^2}.
\end{equation}
Note that, from Lemma \ref{Lem4.3}, we have $\|u_n\|_{L^2}=\|L_n
\zeta\|_{L^2}\to \|\zeta\|_{L^2}=1$ as $n\to \infty$. Hence, by
\eqref{eq04.13}, we get
            \begin{equation}\label{eq4.14}
    \|u_n^+\|_{L^2}\to 1~\text{as}~n\to \infty.
\end{equation}
Now we define the operators $L_n^+:Z\rightarrow
L^2(\mathbb{R}^3,\mathbb{C}^4)$ given by
\begin{equation*}
L_n^+ \zeta=(L_n \zeta)^+,~\zeta\in Z.
\end{equation*}
Thus, \eqref{eq4.14} implies that the operators $L_n^+$ satisfy the
assumption of Lemma \ref{Lem4.1}, so we may assume, without loss of
generality, that $L_n^+$ is injective for every $n$. Now set
    \begin{equation*}
        X_n^k:=L_n^+(Z).
    \end{equation*}
    Then $\dim X_n^k=k$. On the other hand, from \eqref{eq2.10}, one has
    \begin{equation*}
                |((H_0-mI)L_n \zeta,L_n \zeta)_{L^2}|\to 0~\text{as}~n\to \infty,
    \end{equation*}
that is,
    \begin{equation*}
\|u_n^+\|^2-\|u_n^-\|^2-m\|u_n\|_{L^2}\to 0~\text{as}~n\to \infty,
\end{equation*}
which implies that
            \begin{equation*}
\|u_n^+\|\to \sqrt{m}~\text{as}~n\to \infty,
\end{equation*}
     and so $X_n^k\subset V$ for $n$ large. Moreover,
        \begin{equation*}
S_n^k=\left\{\frac{L_n^+\zeta}{\|L_n^+\zeta\|_{L^2}}\bigg|\zeta\in
Z^1\right\}.
    \end{equation*}
We will prove that the sequence $S_n^k$ satisfies all properties
desired. Indeed, for any $v_n\in S_n^k$, there exists $\zeta \in
Z^1$ such that $v_n=\frac{L_n^+\zeta}{\|L_n^+\zeta\|_{L^2}}$. Let
$u_n=L_n\zeta$ again, and we may assume $\frac{1}{2}\leq \|u_n^+\|_{L^2}
\leq 1$, then it follows from \eqref{eq1.3} and the convexity of $\Psi$
that
    \begin{equation}\label{eq4.15}
\Psi(v_n)=\Psi\left(\frac{u_n^+}{\|u_n^+\|_{L^2}}\right)\geq \Psi(u_n^+) \geq
2\Psi\left(\frac{u_n}{2}\right)-\Psi(u_n^-)\geq
2^{1-q}\Psi(u_n)-\Psi(u_n^-).
\end{equation}
By Sobolev inequality, Lemma \ref{Lem4.3} and \eqref{eq4.13}, for
$n$ large, we infer that
    \begin{equation}\label{eq4.16}
\Psi(u_n^-)\leq C\|u_n^-\|^p\leq C \|(H_0-mI)u_n\|_{L^2}^p\leq
2^{-q}\Psi(u_n).
\end{equation}
Thus, \eqref{eq4.15} and \eqref{eq4.16} yields
    \begin{equation}\label{eq4.17}
    \Psi(v_n)\geq 2^{-q}\Psi(u_n).
\end{equation}
On the other hand, we have
    \begin{equation}\label{eq4.18}
    \begin{split}
    \|v_n\|^2-m&=((H_0-mI)v_n,v_n)_{L^2}
        =\frac{((H_0-mI)u_n^+,u_n^+)_{L^2}}{\|u_n^+\|_{L^2}^2}\\
        &\leq 4\big[|((H_0-mI)u_n,u_n)_{L^2}|+((H_0-mI)u_n,u_n^-)_{L^2}\big]
        \\
        &\leq 4\big[|((H_0-mI)u_n,u_n)_{L^2}
        |+\|(H_0-mI)u_n\|_{L^2}\big].
    \end{split}
\end{equation}
Since $L_n$ is injective, jointly with \eqref{eq4.17} and
\eqref{eq4.18}, one has
    \begin{equation*}
        \sup_{v\in S_n^k}\left(\|v\|^2-m\right)\leq 4\sup_{\zeta\in Z^1}\|(H_0-mI)L_n \zeta\|_{L^2}+4\sup_{\zeta\in Z^1}\big|((H_0-mI)L_n \zeta,L_n \zeta)_{L^2}\big|
\end{equation*}
and
    \begin{equation*}
    \inf_{v\in S_n^k}\Psi(v_n)\geq 2^{-q}\inf _{\zeta\in Z^1}\Psi\left(L_n \zeta\right).
\end{equation*}
Hence, Lemma \ref{Lem4.3} implies the conclusion of Lemma
\ref{Lem4.4} holds.
\end{proof}

\begin{Lem}\label{Lem4.5}
    For any $k\in \mathbb{N}$ and $a>0$ small enough, we have
    \begin{equation*}
        b_k<\frac{ma^2}{2}.
    \end{equation*}
\end{Lem}

\begin{proof}
    Let
        $$
S_n^k(a):=\left\{u \in X_n^k \mid\|u\|_{L^2}=a\right\}.
    $$
We shall prove that for any $a>0$ small enough, there exists
$n(a)\in \mathbb{N}$ such that
\begin{equation*}
\sup_{v\in S_n^k(a)}J(v)<\frac{ma^2}{2},~\forall n\geq n(a).
\end{equation*}
Indeed, for any $v_n\in S_n^k(a)$, there exists $z_n\in S_n^k$ with
$v_n=az_n$. In view of Lemma \ref{Lem4.3}, $\{z_n\}$ is uniformly
bounded in $Z^1$. Therefore, similar to the argument of
Lemma \ref{Lem2.3}, we get for $a$ small enough
        \begin{equation*}
J(v_n)\leq \frac{a^2}{2}\|z_n\|^2-2^{1-2q}a^q\Psi(z_n).
    \end{equation*}
Thus,
        \begin{equation*}
\sup_{v\in S_n^k(a)}J(v)\leq  \frac{a^2}{2}\sup_{v\in
S_n^k}\|v\|^2-2^{1-2q}a^q\inf_{v\in S_n^k}\Psi(v).
\end{equation*}
By Lemma \ref{Lem4.4}, there exists $n(a)\in \mathbb{N}$ such that,
for any $n_0\geq n(a)$, there holds
    \begin{equation*}
\sup_{v\in S_{n_0}^k(a)}J(v)<\frac{ma^2}{2}.
    \end{equation*}
Moreover, we also have $S_{n_0}^k(a)\subset X_a$ for $a$ small enough. This completes the
proof.
\end{proof}

\medskip
\noindent{\bf The proof of Theorem 1.2} The multiplicity of Theorem
\ref{Thm1.2} follows from Theorem A, Lemma \ref{Lem3.1} and Lemma
\ref{Lem4.5}. Similar to the arguments of Theorem \ref{Thm1.1}, we
obtain the bifurcation conclusion.






\section*{Acknowledgment}
The second author was supported by NSFC
 12201625. The last author was supported by China Postdoctoral Science Foundation 2022M711853.

\noindent {Annouar Bahrouni\\
Mathematics Department,\\
University of Monastir, Faculty of Sciences, 5019 Monastir, Tunisia\\
e-mail:  bahrounianouar@yahoo.fr}
\medskip
\\
\noindent {Qi Guo\\
School of Mathematics,\\
Renmin University of China, Beijing, 100872, P.R. China\\
e-mail: qguo@ruc.edu.cn}
\medskip
\\
\noindent {Hitchem Hajaiej\\
Department of Mathematics, California State University, Los Angeles,\\
\small 5151 State University Drive,
Los Angeles, CA 90032, USA\\
e-mail: hhajaiej@calstatela.edu}
\medskip
\\
\noindent{Yuanyang Yu\\
Department of Mathematical Sciences, \\
Tsinghua University, Beijing, 100084,  P.R.China\\
e-mail: yyysx43@163.com}

\end{document}